\documentclass[11pt]{amsart}

\usepackage{amsmath}
\usepackage{amssymb}
\usepackage{graphicx}

\newtheorem{theorem}{Theorem}[section]
\newtheorem{corollary}[theorem]{Corollary}
\newtheorem{lemma}[theorem]{Lemma}
\newtheorem{proposition}[theorem]{Proposition}
\theoremstyle{definition}

\newtheorem{remark}[theorem]{Remark}

\numberwithin{equation}{section}

\title[nonlocal parabolic equation]{Finite time blow-up and global solutions for a nonlocal parabolic equation at high energy level}

\author[X. Li]{Xiaoliang Li}
\author[B. Liu]{Baiyu Liu}

\address[X. Li, B. Liu]{School of Mathematics and Physics\\
	University of Science and Technology Beijing \\
	30 Xueyuan Road, Haidian District
	Beijing, 100083\\
	P.R. China}
\email{{\tt liuby@ustb.edu.cn}}



\begin{document}

\begin{abstract}
	In this paper, we consider the solution of a nonlocal parabolic equation. Focusing on the solutions with initial data at high energy level, we find the criteria for global existence and finite time blow up for the corresponding solution respectively. Moreover, we prove that there always exists blow up solution with negative Nehari functional no matter how large the energy is.
\end{abstract}

\maketitle


\section{Introduction}\label{intro}
In this paper, we consider the following initial boundary value problem of nonlocal parabolic equation
\begin{equation}
\label{eq:pro}
\left\{
\begin{array}{ll}
u_t=\Delta u+\left(\frac{1}{|x|^{n-2}}*|u|^p\right)|u|^{p-2}u,& x\in \Omega, t>0\\
u(x,t)=0,& x\in \partial \Omega, t>0\\
u(x,0)=u_0(x),& x\in \Omega
\end{array}
\right.
\end{equation}
where $\Omega$ is a bounded smooth domain of $\mathbb{R}^n$ ($n\geq3$), $1< p<\frac{n+2}{n-2}$ and $\frac{1}{|x|^{n-2}}*|u|^p=\int_{\Omega}\frac{|u(y)|^p}{|x-y|^{n-2}}dy$. This equation can be applied to thermal physics with nonlocal source and model population dynamics, in which cases the individuals are competing not only with others at their own point in space but also with the individual at other points in the domain. For other nonlocal parabolic type equations used in physics and ecology, one can see \cite{Furter, Gourley, Lacey, Ou2007, So2012} and the references therein.

In the sub-critical case  $1<p<(n+2)/(n-2)$, the second author and Ma proved that (\ref{eq:pro}) is locally well-posed in Lebesgue space and possess a variational structure \cite{LiuB2014}. The energy functional
$$
J(u)=\frac{1}{2}\int_\Omega |\nabla u|^2dx-\frac{1}{2p}\int_{\Omega\times \Omega}\frac{|u(y)|^p|u(x)|^p}{|x-y|^{n-2}}dxdy,
$$ 
is nonincreasing along the flow of (\ref{eq:pro}). More precisely, there holds $$\frac{d}{dt}J(u(t))=-\int_{\Omega}u_t^2(t)\,dx,$$ where $u(t)$ is the solution of (\ref{eq:pro}). 
Moreover, using the potential well method \cite{Payne, Sattinger}, in \cite{LiuB2014} the second author and Ma found that for a given low energy initial data $u_0$  (smaller than the depth of potential well), whether the corresponding solution of (\ref{eq:pro}) is global or blow up in finite time is determined by the Nehari functional 
$$
I(u)=(J'(u),u)=\int_\Omega |\nabla u|^2dx-\int_{\Omega}\left(\frac{1}{|x|^{n-2}}*|u|^p\right)|u|^p\,dx.
$$
In \cite{LiLiu2016}, we extended above results to the critical energy level initial data and established the asymptotic behavior results. That is to say, for a regular initial value $u_0\in C(\bar{\Omega})\cap H_0^1(\Omega)$, we have 
\begin{itemize}
	\item If $J(u_0)\leq d$ and $I(u_0)>0$, then the solution of (\ref{eq:pro}) is global and decays to 0 exponentially as $ t\to\infty$.
	\item If $J(u_0)\leq d$ and $I(u_0)<0$, then the solution of (\ref{eq:pro}) blows up in finite time and the norm of it increases exponentially.
\end{itemize}
(See \cite{LiuB2014, LiLiu2016} for more details.)
Here $d$ can be interpreted as the depth of the potential well, i.e.
\begin{equation}
\label{def:d}
d=\inf\{J(u):u\in H_0^1(\Omega)\setminus\{0\},I(u)=0\}. 
\end{equation}

Equation (\ref{eq:pro}) is the nonlocal version of the classical semi-linear parabolic problem
\begin{equation}
\label{eq:sca}
\left\{
\begin{array}{ll}
u_t=\Delta u+|u|^{p-1}u,& x\in \Omega, t>0,\\
u(x,t)=0,& x\in \partial \Omega, t>0,\\
u(x,0)=u_0(x),& x\in \Omega,
\end{array}
\right.
\end{equation}
which has been studied extensively (see \cite{Filip2005, Dick2011, Payne, Quittner2003} and the references therein). In the lower energy level, whether the solution of (\ref{eq:sca}) is global existence or blows up in finite is totally determined by the Nehari functional. While in the high energy level,  Gazzola and Weth \cite{Filip2005}, Dickstein et.al \cite{Dick2011} have shown that either positive or negative Nehari functional is not sufficient for finite time blowup or global existence of the solutions of (\ref{eq:sca}). They also give some criteria for the solution to be global or blow up in finite at high energy level.

It is natural to seek the sufficient conditions for the solution of (\ref{eq:pro}) to be global or blow up in finite at high energy level $J(u_0)>d$, that is also the motivation of this present paper.

To state our main results, we introduce some notions.  
Let
\begin{align*}
&\mathcal{G}_0 =\{u_0\in H_0^1(\Omega):T_{\max}(u_0)=\infty \quad \mathrm{and}\quad u(t)\to 0, \textrm{as}\ t\to \infty \},\\
&\mathcal{B} =\{u_0\in H_0^1(\Omega): T_{\max}(u_0)<\infty \},\\
&\mathcal{N}:=\{u\in H_0^1(\Omega)\backslash \{0\}:I(u)=0\}.
\end{align*}
We define $$\mathcal{N}_{+}=\{u\in H_0^1(\Omega):I(u)>0\}\ \mathrm{and}\ \mathcal{N}_{-}=\{u\in H_0^1(\Omega):I(u)<0\}.$$
For a given number $a>d$ ($d$ is defined as in (\ref{def:d})), define $$\mathcal{N}_a:=\{u\in H_0^1(\Omega)\backslash \{0\}:I(u)=0, J(u)<a\}.$$ 
Note that
\begin{equation}
\label{eq:J-I}
J(u)=\frac{1}{2p}I(u)+\left(\frac12-\frac{1}{2p}\right)||\nabla u||^2,
\end{equation}
which gives $$\mathcal{N}_a\equiv\Big{\{}u\in H_0^1(\Omega)\backslash \{0\}:I(u)=0,||\nabla u||<\sqrt{\frac{2pa}{p-1}}\Big{\}}.$$
Denote $$\lambda_a=\inf_{u\in\mathcal{N}_a}||u||,\quad \Lambda_a=\sup_{u\in\mathcal{N}_a}||u||.$$

Our main result shows that for those initial datum with small $L^2$ norms and the associated Nehari functionals are positive, the corresponding solutions of problem (\ref{eq:pro}) are global, while for those initial datum that are sufficiently large and the associated Nehari functionals are negative, the solutions will blow up in finite time. 
\begin{theorem}
	\label{thm:criterions}
	Let $\Omega$ be a smooth bounded convex domain and $1<p<p_{nl}:=\frac{2n}{(n-1)(n-2)},\quad n\geq 3$. If $u_0\in C(\bar{\Omega})\cap\mathcal{N}_{+}$ and $||u_0||\leq\lambda_{J(u_0)}$, then $u_0\in\mathcal{G}_0$. If $u_0\in C(\bar{\Omega})\cap\mathcal{N}_{-}$ and $||u_0||\geq\Lambda_{J(u_0)}$, then $u_0\in\mathcal{B}$.
\end{theorem}

The following result exhibits the existence of a class of initial data in region $\mathcal{N}_{-}\cap\mathcal{B}$ with arbitrarily high energy which gives rise to blow up.
\begin{theorem}
	\label{thm:blow-up for M}
	Let $\Omega$ be a smooth bounded convex domain and $2<p<3$, $n=3$.
	For any $M>0$ there exists $u_M\in \mathcal{N}_{-}$ such that $J(u_M)\geq M$ and $u_M\in\mathcal{B}$.
\end{theorem}

The remainder of the paper is organized as follows. In the next section, we establish the existence of positive stationary solutions of problem (\ref{eq:pro}). Some preliminaries for the global solutions will be given in section \ref{sec:4}. In section \ref{sec:5}, we give the proofs of Theorem \ref{thm:criterions} and Theorem \ref{thm:blow-up for M}.

Throughout the paper, we assume $\Omega\subset\mathbb{R}^n,\  n\geq 3$ is a bounded domain of class $C^{2+\alpha}$ for some $\alpha\in (0,1)$. We denote $z(u)=\frac{1}{|x|^{n-2}}*|u|^p$, $||\cdot ||_p=||\cdot ||_{L^p(\Omega)}, ||\cdot||=||\cdot||_2$ and $T_{\max}$ is the maximal existence time of solution of problem (\ref{eq:pro}), and let $$p_{nl}:=\frac{2n}{(n-1)(n-2)},\quad n\geq 3.$$

\section{The existence of positive stationary solution}\label{sec:3}

In this section, we consider the stationary problem of (\ref{eq:pro})
\begin{equation}
\label{eq:spro}
\left\{
\begin{array}{ll}
-\Delta u=\left(\frac{1}{|x|^{n-2}}*|u|^p\right)|u|^{p-2}u,& x\in \Omega,\\
u(x)=0,& x\in \partial \Omega.
\end{array}
\right.
\end{equation}
We obtain the following existence result of the positive stationary solution.
\begin{theorem}
	\label{thm:solu-exist}
	Assume $1< p<p_{nl}:=\frac{2n}{(n-1)(n-2)}$, then there exists a positive classical solution $u\in C^2(\bar{\Omega})$ for stationary problem (\ref{eq:spro}). 
\end{theorem}

The following two lemmas will be used in our proof. We include them here for the readers' convenience. The first lemma is a $L^q$-estimate for the nonlinear term $z(u)|u|^{p-2}u$, where $z(u)=\frac{1}{|x|^{n-2}}*|u|^p$. 
\begin{lemma}[Lemma 4 in \cite{LiuB2014}]
	\label{lem:estimation-f}
	For $q\geq n-1$, then nonlinear term $z(u)|u|^{p-2}u$ in (\ref{eq:pro}) satisfies the following estimates.
	\begin{enumerate}
		\item[(i)] If $1<p<2$, then there is a constant $c=c(\Omega,p,q)$ such that for all $u,v\in L^{pq}(\Omega)$
		\begin{align*}
		||z(u)|u|^{p-2}u-z(v)|v|^{p-2}v||_q&\leq c(\Omega,p,q)\{(||u||_{pq}^{2p-2}+\\
		&||u||_{pq}^{p-1}||v||_{pq}^{p-1})||u-v||_{pq}+||v||_{pq}^p||u-v||_{pq}^{p-1}\}.
		\end{align*}
		\item[(ii)] If $p\geq2$, then there is a constant $c=c(\Omega,p,q)$ such that for all $u,v\in L^{pq}(\Omega)$
		\begin{align*}
		||z(u)|u|^{p-2}u-z(v)|v|^{p-2}v||_q&\leq c(\Omega,p,q)\{(||u||_{pq}^{2p-2}+||u||_{pq}^{p-1}||v||_{pq}^{p-1}+\\
		&||u||_{pq}^{p-2}||v||_{pq}^{p}+||v||_{pq}^{2p-2})||u-v||_{pq}+\\
		&||v||_{pq}^p||u-v||_{pq}^{p-1}\}.
		\end{align*}
	\end{enumerate}
\end{lemma}
The second one is the well-known Lagrange multiplier rule which will help us to find the critical point of $J(u)$. 
\begin{lemma}[Theorem 6.1 in \cite{Quittner2007}]
	\label{pro:L-m}
	Let $X$ be a real Banach space, $w\in X$ and let $\Psi,\Phi_1,\dots,\Phi_k:X\to\mathbb{R}$ be $C^1$ in a neighborhood of $w$. Denote $M:=\{u\in X:\Phi_i(u)=\Phi_i(w)\  \mathrm{for}\ i=1,\dots,k\}$ and assume that $w$ is a local minimizer of $\Psi$ with respect to the set $M$. If $\Phi_1'(w),\dots,\Phi_k'(w)$ are linearly independent, then there exist $\mu_1,\dots,\mu_k\in\mathbb{R}$ such that $$\Psi'(w)=\sum_{i=1}^k\mu_i\Phi_i'(w).$$
\end{lemma}

We now present the proof of Theorem \ref{thm:solu-exist}, the idea of which is to find a critical point of energy function $J(u)$. Furthermore, by the elliptic regularity for linear equations \cite{Quittner2007, Gilbarg1998} we show that the variational solution of (\ref{eq:spro}) is also a classical solution which belongs to $C^2(\bar{\Omega})$. 

\begin{proof}[Proof of Theorem \ref{thm:solu-exist}]
	We divide our proof into three steps.
	
	$\mathbf{Step \,1}$. The energy functional $J(u)$ belongs to $C^1(H_0^1(\Omega);\mathbb{R})$. 
	
	For a start, we denote the norm by $||u||_{H_0^1(\Omega)}:=||\nabla u||$ and the inner product in the space $H_0^1(\Omega)$ by $(u,v):=\int_{\Omega}\nabla u\cdot\nabla v\,dx$. And we write the energy functional $J(u)=\Psi(u)-\Phi(u)$, where
	$$\Psi(u):=\frac12||\nabla u||^2\quad \mathrm{and}\quad \Phi(u):=\frac{1}{2p}\int_{\Omega\times \Omega}\frac{|u(y)|^p|u(x)|^p}{|x-y|^{n-2}}dxdy .$$
	
	In view of Lemma \ref{lem:estimation-f}, we deduce that 
	\begin{equation}
	\label{eq:estimation-f-p}
	||z(u)|u|^{p-2}u||_q\leq C(\Omega,p)||u||_{pq}^{2p-1}
	\end{equation}
	for $q\geq n-1\geq 2$. 
	For $1< p<\frac{2n}{(n-1)(n-2)}$, by choosing a $q\geq n-1$ such that $pq\leq2n/(n-2)$ and using Sobolev embedding theorem, we have $z(u)|u|^{p-2}u\in L^q(\Omega)\subset L^2(\Omega)\hookrightarrow H^{-1}(\Omega)$ for $u\in H_0^1(\Omega)$. Therefore, put $f(u):=z(u)|u|^{p-2}u$, we compute 
	\begin{equation}
	\label{eq:derivative-J}
	R[J'(u)]=R[\Psi'(u)]-R[\Phi'(u)]=u-R[f(u)],
	\end{equation}
	where $R:H^{-1}(\Omega)\to H_0^1(\Omega)$ is the Riesz isometric isomorphism. Apparently, $\Psi(u)\in C^1(H_0^1(\Omega);\mathbb{R})$ and we need only to prove $\Phi(u)\in C^1(H_0^1(\Omega);\mathbb{R})$ too. To see this, notice that if $u,v\in H_0^1(\Omega)$, then by Lemma \ref{lem:estimation-f} we obtain
	\begin{align*}
	||\Phi'(u)-\Phi'(v)||_{H^{-1}(\Omega)}
	&=||f(u)-f(v)||_{H^{-1}(\Omega)}\\
	&\leq C(n,p,\Omega)||f(u)-f(v)||_q\\
	&\leq C(n,p,\Omega)\left(||u-v||_{pq}+||u-v||_{pq}^{p-1}\right).
	\end{align*}
	Thus, since $pq\leq2n/(n-2)$, the mapping $\Phi'(u):H_0^1(\Omega)\to \mathbb{R}$ is continuous. Consequently, $\Phi(u)\in C^1(H_0^1(\Omega);\mathbb{R})$, and we find $J(u)\in C^1(H_0^1(\Omega);\mathbb{R})$.
	
	$\mathbf{Step \,2}$. We shall show there is a critical point of $J(u)$ which is a nonnegative variational solution of (\ref{eq:spro}).
	
	Let $u_k\in M:=\{u\in H_0^1(\Omega):\Phi(u)=1\}$, we claim the set $M$ is weakly sequentially closed in $H_0^1(\Omega)$. To see this, let $u_k\rightharpoonup u$ in $H_0^1(\Omega)$, we need to show $\Phi(u_k)\to\Phi(u)$ and we estimate
	$$|\Phi(u_k)-\Phi(u)|=\left|\int_{\Omega}z(u_k)|u_k|^p\,dx-\int_{\Omega}z(u)|u|^p\,dx\right|\leq I_1+I_2$$
	where $$I_1:=\left|\int_{\Omega}z(u_k)(|u_k|^p-|u|^p)\,dx\right|,\quad I_2:=\left|\int_{\Omega}(z(u_k)-z(u))|u|^p\,dx\right|.$$
	Using the H\"older inequality, for $q\geq n-1,1<q'\leq (n-1)/(n-2)$ such that $pq<2n/(n-2)$ and $1/q+1/q'=1$, we have 
	$$z(u_k)=\int_{\Omega}\frac{|u_k(y)|^p}{|x-y|^{n-2}}\,dy\leq ||u_k||_{pq}^p\left(\int_{\Omega}\frac{1}{|x-y|^{(n-2)q'}}\,dy\right)^{1/q'}\leq C_{\Omega}||u_k||_{pq}^p.$$
	Since $\{u_k\}$ is bounded in $H_0^1(\Omega)$ we obtain $z(u_k)$ is also bounded. And, we get $u_k\to u$ in $L^{pq}(\Omega)$ due to $H_0^1(\Omega)\hookrightarrow L^{pq}(\Omega)$. This implies $I_1\to 0$ as $k\to\infty$. Then in order to prove $I_2\to 0$, we estimate similarly
	\begin{align*}
	|z(u_k)-z(u)|&=\left|\int_{\Omega}\frac{|u_k(y)|^p-|u(y)|^p}{|x-y|^{n-2}}\,dy\right|\\
	&\leq C_{\Omega,p,q}\left(\int_{\Omega}||u_k|^p-|u|^p|^q\,dy\right)^{1/q}\\
	&\leq C_{\Omega,p,q}\left(\int_{\Omega}\xi^{(p-1)q}||u_k|-|u||^q\,dy\right)^{1/q},
	\end{align*}
	here $\xi(y)$ is between $|u_k(y)|$ and $|u(y)|$. Hence
	\begin{align*}
	|z(u_k)-z(u)|&\leq C_{\Omega,p,q}\left(\int_{\Omega}\xi^{(p-1)q\frac{p}{p-1}}\,dy\right)^{\frac{p-1}{pq}}\left(\int_{\Omega}||u_k|-|u||^{pq}\,dy\right)^{\frac{1}{pq}}\\
	&\leq  C_{\Omega,p,q}(||u_k||_{pq}^{p-1}+||u||_{pq}^{p-1})||u_k-u||_{pq}.
	\end{align*}
	Above estimate means $z(u_k)\to z(u)$ as $k\to\infty$, and so $I_2\to 0$ as $k\to\infty$. Consequently, $\Phi(u_k)\to\Phi(u)$ and hence $u\in M$, this proves our claim.
	
	Since $$\Psi''(u)[h,h]=2\Psi(h)=||\nabla h||^2\geq 0,$$
	the functional $\Psi$ is convex and coercive, it follows from the result of existence of minimizer that there exists $w\in M$ such that $\Psi(w)=\inf_M\Psi$. Moreover, notice that $|w|\in M$ due to $\Psi(|w|)=\Psi(w)$, we may assume that $w\geq 0$. Since $$\Phi'(w)w=2p\Phi(w)=2p,$$ 
	it implies $\Phi'(w)\neq 0$. According to Lemma \ref{pro:L-m}, there exists $\mu$ such that $\Psi'(w)=\mu \Phi'(w)$, hence $$0<2\Psi(w)=\Psi'(w)w=\mu\Phi'(w)w=\mu2p,$$
	which means $\mu>0$. Setting now $t:=\mu^{1/(2p-2)}$, we get a critical point:
	$$J'(tw)=\Psi'(tw)-\Phi'(tw)=t(\Psi'(w)-t^{2p-2}\Phi'(w))=0.$$
	Therefore, $u:=tw=\mu^{1/(2p-2)}w\not\equiv0$ is a nonnegative variational solution of elliptic problem (\ref{eq:spro}).

	
	$\mathbf{Step \,3}$. We now show that the variational solution $u$ obtained above satisfies $u\in C^2(\bar{\Omega})$.
	
	In fact, this is a consequence of standard regularity results for linear elliptic equations. To see this, by a simple bootstrap argument, we first claim $\tilde{f}(x):=z(u(x))|u(x)|^{p-2}u(x)\in L^q(\Omega)$ for any $2\leq q<\infty$ as follows. 
	
	Notice $p<\frac{2n}{(n-1)(n-2)}\leq \frac{n}{n-2}$ for $n\geq 3$, and fix $\rho\in(1,n/(n-2)p)$. We assume that there holds
	\begin{equation}
	\label{eq:boot-f}
	\tilde{f}\in L^{\rho^i}(\Omega)
	\end{equation}
	for some $i\geq0$ (this is true for $i=0$ by fact that $\tilde{f}\in L^2(\Omega)\subset L^1(\Omega)$ from (\ref{eq:estimation-f-p})). Recall that the variationl solution $u$ is also an $L^1$-solution, and since
	$$
	\frac{1}{\rho^i}-\frac{1}{p\rho^{i+1}}=\frac{1}{\rho^i}\left(1-\frac{1}{p\rho}\right)<\frac{2}{n},
	$$
	by using the Laplacian's regularity for $L^1$-solution (see for example Proposition 47.5(i) in \cite{Quittner2007}), we obtain $u\in L^{p\rho^{i+1}}(\Omega)$, hence $\tilde{f}\in L^{\rho^{i+1}}(\Omega)$ follows from (\ref{eq:estimation-f-p}). Thus, by induction, it follws that (\ref{eq:boot-f}) is true for all integers $i$ and proves our claim.
	
	Then, we may apply the elliptic regularity (see Theorem 47.3(i) in \cite{Quittner2007}) to deduce the existence of $\tilde{u}\in W_0^{1,q}(\Omega)$ such that $-\Delta\tilde{u}=\tilde{f}$. Since $u,\tilde{u}\in H_0^1(\Omega)$, the maximum principle for variational solution (see Proposition 52.3(i) in \cite{Quittner2007}) yields that $u=\tilde{u}$. Hence, due to the embedding $W^{1,q}(\Omega)\hookrightarrow C^{\alpha}(\bar{\Omega})$ for $q>n$, we deduce that $u\in C^{\alpha}(\bar{\Omega})$, where $\alpha=1+[n/q]-n/q$. 
	
	Now, in order to prove $u\in C^2(\bar{\Omega})$ by using elliptic regularity (see Theorem 6.14 in \cite{Gilbarg1998}), we prove $\tilde{f}\in C^{\alpha}(\bar{\Omega})$ in the following way.
	
	Let $|\cdot|_{\alpha;\bar{\Omega}}:=||\cdot||_{C(\bar{\Omega})}+[\cdot]_{\alpha;\bar{\Omega}}$ be the norm in $C^{\alpha}(\bar{\Omega})$. Then  
	\begin{equation}
	\label{eq:holder-f}
	|\tilde{f}|_{\alpha;\bar{\Omega}}=|z(u)|u|^{p-2}u|_{\alpha;\bar{\Omega}}\leq |z(u)|_{\alpha;\bar{\Omega}}||u|^{p-2}u|_{\alpha;\bar{\Omega}}.
	\end{equation}
	Since $u\in C^{\alpha}(\bar{\Omega})$, we have 
	\begin{align*}
	[|u|^p]_{\alpha;\bar{\Omega}}=\sup_{x\neq y}\frac{||u(x)|^p-|u(y)|^p|}{|x-y|^{\alpha}}&=\sup_{x\neq y}p\xi^{p-1}\frac{||u(x)|-|u(y)||}{|x-y|^{\alpha}}\\
	&\leq C(\Omega,p)[u]_{\alpha;\bar{\Omega}}\leq C,
	\end{align*}
	where $\xi$ is between $|u(x)|$ and $|u(y)|$. Using this, we derive
	\begin{align*}
	[z(u)]_{\alpha;\bar{\Omega}}&=\sup_{x\neq y}\frac{|z(u)(x)-z(u)(y)|}{|x-y|^\alpha}\\
	&=\sup_{x\neq y}\left|\int_{\Omega}\frac{|u(x-z)|^p-|u(y-z)|^p}{(|x-y|^{\alpha})|z|^{n-2}}\,dz\right|\\
	&\leq [|u|^p]_{\alpha;\bar{\Omega}}\int_{\Omega}\frac{1}{|z|^{n-2}}\,dz\\
	&=C(\Omega) [|u|^p]_{\alpha;\bar{\Omega}}\leq C.
	\end{align*}
	Thus we get 
	\begin{equation}
	\label{eq:holder-z}
	|z(u)|_{\alpha;\bar{\Omega}}=||z(u)||_{C(\bar{\Omega})}+[z(u)]_{\alpha;\bar{\Omega}}\leq C.
	\end{equation}
	Similarly,  we have 
	\begin{equation}
	\label{eq:holder-u}
	||u|^{p-2}u|_{\alpha;\bar{\Omega}}=|||u|^{p-2}u||_{C(\bar{\Omega})}+[|u|^{p-2}u]_{\alpha;\bar{\Omega}}\leq C(\Omega)+C(\Omega,p)[u]_{\alpha;\bar{\Omega}}\leq C.
	\end{equation}
	Combining (\ref{eq:holder-f}),  (\ref{eq:holder-z}) and (\ref{eq:holder-u}), we get $\tilde{f}\in C^{\alpha}(\bar{\Omega})$. 
	
	Therefore, applying Theorem 6.14 in \cite{Gilbarg1998}, we deduce that $u\in C^2(\bar{\Omega})$. Furthermore, the strong maximum principle shows $u>0$ in $\Omega$.
	
	This completes the proof.
\end{proof}

\section{Boundness and asymptotic property of global solutions}\label{sec:4}

This section is devoted to the global solution of problem (\ref{eq:pro}). We find that the energy functional of global solution is nonnegative in the whole time. Meanwhile, we give a $L^2$-norm bound determined by the initial energy for global solution. Moreover, we prove that each global solution of (\ref{eq:pro}) will convergent to an equilibrium. 

\begin{proposition}
	\label{pro:bound-J}
	Let $\Omega$ be a smooth bounded convex domain and $1<p<\frac{n+2}{n-2} \  (n\geq3)$. Assume $u_0\in C(\bar\Omega)\cap H_0^1(\Omega)$ and $T_{\max}(u_0)=\infty$. Then the corresponding solution $u(t)$ of (\ref{eq:pro}) fulfills $0\leq J(u(t))\leq J(u_0)$ and 
	\begin{equation}
	\label{eq:estimation-u}
	||u(t)||^2\leq \frac{2p}{(p-1)\lambda}J(u_0)
	\end{equation}
	for all $t>0$, where $\lambda$ is the first eigenvalue of $-\Delta$ in $H_0^1(\Omega)$.
\end{proposition}

\begin{proof}
	For any $t>0$, it's known that $\frac{d}{dt}J(u(t))=-\int_{\Omega}u_t^2\,dx$ and $J(u(t))\leq J(u_0)$. To prove $J(u(t))\geq 0$, we assume for contradiction that there exits a $t'>0$ such that $J(u(t'))<0$. Since $$I(u(t'))=2pJ(u(t'))-(p-1)||\nabla u(t')||^2\leq2pJ(u(t'))<0$$ and $J(u(t'))<0<d$, recall the result of blow-up solutions in the low energy case (see Section \ref{intro}), we deduce $T_{\max}(u(t'))<\infty$ immediately, which contradicts $T_{\max}(u_0)=\infty$. Therefore, we get $0\leq J(u(t))\leq J(u_0)$ for all $t>0$.
	
	Then, to see (\ref{eq:estimation-u}), we denote $M(t)=\frac12\int_0^t||u(\tau)||^2d\tau$, and $M'(t)=\frac12||u(t)||^2$. Using the Poincar\'e inequality we have
	\begin{align}
	M''(t)&=\int_{\Omega}uu_t\,dx=-I(u(t))\notag\\
	&=-2pJ(u(t))+(p-1)\int_{\Omega}|\nabla u(t)|^2\,dx\notag\\
	&\geq -2pJ(u_0)+(p-1)\lambda\int_{\Omega}u(t)^2\,dx.\label{eq:L2norm-u}
	\end{align}
	By estimate (\ref{eq:L2norm-u}), we claim that $$(p-1)\lambda||u(t)||^2\leq2pJ(u_0),\quad t>0.$$
	If otherwise, $\exists\, t''>0$ s.t. $(p-1)\lambda||u(t'')||^2>2pJ(u_0)$. Using (\ref{eq:L2norm-u}), we get 
	$$M''(t'')\geq -2pJ(u_0)+2(p-1)\lambda M'(t'')>0.$$
	Thus, let $C_0:=-2pJ(u_0)+2(p-1)\lambda M'(t'')$, we see
	$$ M''(t)\geq C_0>0, \quad\forall t\geq t'',$$
	which indicates $M'(t)\to \infty$ and $M(t)\to\infty$ as $t\to\infty$. Now  for contradiction, we estimate
	\begin{equation}
	\label{eq:M'-M''}
	M''(t)\geq 2p\int_0^t||u_t||^2d\tau+2(p-1)\lambda M'-2pJ(u_0),\quad t>0.
	\end{equation}
	Integrating $M''(t)=\int_{\Omega}uu_tdx$ on $(0,t)$ yields
	$$M'(t)-M'(0)=\int_0^t\int_{\Omega}uu_t\,dxd\tau,$$
	then
	\begin{align}
	(M'(t))^2&=-(M'(0))^2+2M'(t)M'(0)+\left(\int_0^t\int_{\Omega}uu_tdxd\tau\right)^2 \notag \\
	&=-\frac14||u_0||^2+M'(t)||u_0||^2+\left(\int_0^t\int_{\Omega}uu_tdxd\tau\right)^2\notag \\
	&\leq M'(t)||u_0||^2+\left(\int_0^t\int_{\Omega}uu_tdxd\tau\right)^2 \label{eq:esti-M'}.
	\end{align}
	Hence, combining (\ref{eq:M'-M''}) and (\ref{eq:esti-M'}), we have
	\begin{align}
	MM''-pM'^2 &\geq p\left[ \int_0^t||u||^2d\tau \cdot \int_0^t||u_t||^2d\tau -\left(\int_0^t\int_{\Omega}uu_tdxd\tau\right)^2\right] \notag \\
	&+(p-1)\lambda MM'-pMJ(u_0)-pM'||u_0||^2 \notag \\
	&\geq (p-1)\lambda MM'-pMJ(u_0)-pM'||u_0||^2 \label{eq:M-M'-M''},
	\end{align}
	where we have used Schwatz's inequality. Due to $M'(t)\to \infty$ and $M(t)\to\infty$ as $t\to\infty$, we may choose a $t_0>t''$ such that
	$$ \frac{p-1}{2}\lambda M(t)>p||u_0||^2 ,\ \frac{p-1}{2}\lambda M'(t)>pJ(u_0) , \, t>t_0.$$
	Thus, from (\ref{eq:M-M'-M''}) we get
	$$ M(t)M''(t)-pM'(t)^2>0, \, t>t_0.$$
	This inequality guarantees that the function $M^{1-p}(t)$ is nonincreasing and concave on $[t_0,\infty)$. Therefore, there exists a finite time $T>t_0$ such that $\lim_{t\to T}M^{1-p}(t)=0$ i.e. $\lim_{t\to T}M(t)=+\infty$ which contradicts the fact $T_{\max}(u_0)=\infty$.
	
	Consequently, (\ref{eq:estimation-u}) holds by our claim, and we complete the proof.
\end{proof}

The following result is a direct consequence of Proposition \ref{pro:bound-J} and Hardy-Littlewood-Sobolev inequality.
\begin{corollary}
	\label{cor:bound}
	Let $\Omega$ be a smooth bounded convex domain and $1<p\leq\frac{n+2}{n} \  (n\geq3)$. Assume $u_0\in C(\bar\Omega)\cap H_0^1(\Omega)$ and $T_{\max}(u_0)=\infty$. Then corresponding solution $u(t)$ of (\ref{eq:pro}) is bounded in $H_0^1(\Omega)$ for all $t>0$.
\end{corollary}
\begin{proof}
	By using the Hardy-Littlewood-Sobolev inequality we have
	$$\frac12||\nabla u(t)||^2=J(u(t))+\frac{1}{2p}\int_{\Omega}z(u)|u|^p\,dx\leq C+\frac{1}{2p}C_{n,p}||u(t)||_{\frac{2np}{n+2}}^{2p}.$$
	Since $1<p\leq(n+2)/n$, $2np/(n+2)\leq2$ and so $||u(t)||_{2np/(n+2)}\leq C_{n,p,\Omega}||u(t)||$ by H\"older inequality. In view of (\ref{eq:estimation-u}), we immediately deduce that $$||\nabla u(t)||\leq C, \quad t\geq 0,$$ 
	where $C$ is a constant which depends on $J(u_0)$. Hence, any global solution of (\ref{eq:pro}) is bounded in $H_0^1(\Omega)$.
\end{proof}

\begin{remark}
	From the proof of Corollary \ref{cor:bound}, provided $1<p<\frac{n+2}{n}$, we actually can further look for an a priori estiamte of global solutions for problem (\ref{eq:pro}). Indeed, given $K>0$, if $||\nabla u_0||\leq K$, then $J(u_0)\leq \frac12 K^2$. This implies there exists constant $C(K)$ such that all global solutions satisfy $||\nabla u(t)||\leq C(K)$ for all $t\geq 0$.
\end{remark}

Next we study the dynamical behavior of system generated by (\ref{eq:pro}). Let $u(t)$ be a local solution to problem (\ref{eq:pro}) on $[0,T_{\max}(u_0))$. We denote by $S(t): u_0\mapsto u(t)$ be a dynamical system corresponding to (\ref{eq:pro}). Thus, instead of $u(t)$ we will also write $S(t)u_0$ for $t<T_{\max}(u_0)$. For global solution $u(t)$ with initial datum $u_0$, define the $\omega$-limit set $\omega(u_0)$ by
$$\omega(u_0):=\{\varphi\in H_0^1(\Omega)|\,\exists\,t_n\to\infty\ \mathrm{s.t.}\ u(t_n)\to\varphi\ \mathrm{as}\ n\to\infty\}.$$
Using the above definition, we have
\begin{theorem}
	\label{thm:equilibria}
	Let $\Omega$ be a smooth bounded convex domain and $1<p<p_{nl}$. Assume $u_0\in C(\bar\Omega)\cap H_0^1(\Omega)$ and $T_{\max}(u_0)=\infty$. Then:
	\begin{enumerate}
		\item[(i)] The trajectory $\{S(t)u_0:t\geq 0\}$ is relatively compact in $H_0^1(\Omega)$.
		\item[(ii)] The $\omega$-limit set $\omega(u_0)$ is a nonempty comapct and connected subset of $H_0^1(\Omega)$ which consists of solutions of (\ref{eq:spro}).
	\end{enumerate}
\end{theorem}

Above result is a technical consequence in the study of partial differential equations, and there are similar results for more general parabolic equations,  see \cite{Chill2006, Carm1999, Fuji1969, Quittner1994, Quittner2003} and the references therein. 
\begin{proof}[Proof of Theorem \ref{thm:equilibria}]
	Throughout the proof, we denote $f(u):=z(u)|u|^{p-2}u$ and $u(t):=S(t)u_0$. Also, notice that $u(t)\in C^1((0,\infty), L^2(\Omega))\cap  C([0,\infty), H_0^1(\Omega))$ (see \cite{LiuB2014}).
	
	(i). For arbitrary sequence $\{S(t_n)u_0\}\in\{S(t)u_0:t\geq 0\}, n\in\mathbb{N}$, if series $\{t_n\}$ is bounded, then it's clear that there exists a $t'<\infty$ such that $t_n\to t', S(t_n)u_0\to S(t')u_0\in H_0^1(\Omega)$. In the following, given any unbounded time series  $\{t_n\}$, let us show the set $\{S(t_n)u_0\}$ has a convergent subsequence in $H_0^1(\Omega)$.
	
	 For global solution $u(t)$, we have $0\leq J(u(t))\leq J(u_0)$ from the proof of Proposition \ref{pro:bound-J}. Hence, 
	$$\int_0^{\infty}\int_{\Omega}u_t^2\,dxdt=J(u_0)-\lim_{t\to\infty}J(u(t))\leq J(u_0).$$ 
	By this, we deduce that there is a subsequence(we denote still by $\{t_n\}$ for convenience) of $\{t_n\}$ such that $t_n\to\infty$ and $\int_{\Omega} u_t^2(t_n)\,dx\to 0$ as $n\to\infty$, hence $\lim_{t_n\to\infty}u_t(t_n)=0$ a.e. in $\Omega$. Then along the flow (\ref{eq:pro}), we compute 
	$$(J'(u(t)),v)=-\int_{\Omega} u_tvdx,\quad t>0$$
	for any $v\in H_0^1(\Omega)$. Therefore, the Cauchy-Schwarz inequality gives $$J'(u(t_n))\to 0,\quad n\to\infty.$$ That is, for each $\epsilon>0$ and $v\in H_0^1(\Omega)$ we have 
	$$
	|(J'(u(t_n)),v)|=\left|\int_{\Omega}\nabla u(t_n)\cdot\nabla v-f(u(t_n))v\,dx\right|\leq \epsilon||\nabla v||
	$$
	for $n$ sufficiently large. Let $v=u(t_n)$ above to find
	$$
	\left|\int_{\Omega}|\nabla u(t_n)|^2-f(u(t_n))u(t_n)\,dx\right|\leq \epsilon||\nabla u(t_n)||.
	$$
	Take $\epsilon=1$ above in particular, for all $n$ sufficiently large, we see that
	\begin{equation}
	\label{eq:estimate-f-u}
	\int_{\Omega}f(u(t_n))u(t_n)\,dx\leq||\nabla u(t_n)||^2+||\nabla u(t_n)||.
	\end{equation}
	By this, we now claim the sequence $\{u(t_n)\}$ is bounded in $H_0^1(\Omega)$. Since
	$$J(u(t_n))=\frac12||\nabla u(t_n)||^2-\frac{1}{2p}\int_{\Omega}f(u(t_n))u(t_n)\,dx\leq J(u_0)<\infty,$$
	using (\ref{eq:estimate-f-u}), we deduce
	\begin{align*}
	||\nabla u(t_n)||^2&\leq 2J(u_0)+\frac{1}{p}\int_{\Omega}f(u(t_n))u(t_n)\,dx\\
	&\leq 2J(u_0)+\frac1p(||\nabla u(t_n)||^2+||\nabla u(t_n)||).
	\end{align*}
	This implies $||\nabla u(t_n)||^2\leq C$, and so $\{u(t_n)\}$ is indeed bounded.  
	
	Next, in view of (\ref{eq:derivative-J}), we obatin 
	\begin{equation}
	\label{eq:Jn-zero}
	R[J'(u(t_n))]=u(t_n)-R[f(u(t_n))]\to 0\ \ \mathrm{in}\ \,H_0^1(\Omega).
	\end{equation} 
	As proved in Theorem \ref{thm:solu-exist}, we know $f(u)\in L^2(\Omega)$ is continuous in $H_0^1(\Omega)$. Since the embedding $L^2(\Omega)\hookrightarrow H^{-1}(\Omega)$ is compact, the mapping: $u\mapsto R[f(u)]$ is also compact. Thus, by boundness of  $\{u(t_n)\}$, we may assume (passing to a subsequence if necessary) $R[f(u(t_n))]\to \varphi$ for some $\varphi\in H_0^1(\Omega)$. Then $u(t_n)\to \varphi$ due to (\ref{eq:Jn-zero}). Therefore, we find that the trajectory $\{S(t)u_0:t\geq 0\}$ is relatively compact in $H_0^1(\Omega)$.
	
	(ii). The first part of assertion (ii) is a standard consequence from (i) in the study of Dynamical systems (see Appendix G in \cite{Quittner2007}). More precisely, given the trajectory $\{S(t)u_0:t\geq 0\}$ is relatively compact in $H_0^1(\Omega)$, it is obvious that the $\omega$-limit set $\omega(u_0)$ is nonempty. Then using Proposition 53.3 in \cite{Quittner2007} we further have $\omega(u_0)$ is compact and connected. 
	
	We now prove the rest of assertion (ii). Assume $\varphi\in\omega(u_0)$, i.e. $\exists\, t_n\to\infty$ s.t. $u(t_n)\to\varphi$, we next show $\varphi$ is a solution of (\ref{eq:spro}). 
	
	Note that $J(u(t))$ is nonincreasing and bounded, hence $e:=\lim_{t\to\infty}J(u(t))$ exists, and then the continuity of $J(\cdot)$ implies $J(u(t_n))\to J(\varphi)=e$. Put $u(t_n+t):=S(t)u(t_n)$. Since for every $t\geq0$, 
	$$\int_0^t||u_t(t_n+s)||^2\,ds=J(u(t_n))-J(u(t_n+t))\to 0,$$
	we thus deduce that 
	$$u(t_n+t)=u(t_n)+\int_0^t u_t(t_n+s)\,ds\to \varphi$$
	as $t_n\to\infty$. It follows that  $S(t)\varphi=\varphi$, which implies $\varphi$ is a solution of problem (\ref{eq:spro}).
	
	This completes the proof. 
\end{proof}

\section{Solutions at high energy level}\label{sec:5}

In this section, we deal with the solution of (\ref{eq:pro}) at high energy level. 
Before giving the proof of Theorem \ref{thm:criterions}, we prepare several Lemmas.

\begin{lemma}
	Assume $1<p<p_{nl}$ and $a>d$, then $0<\lambda_a\leq\Lambda_a<\infty$.
\end{lemma}

\begin{proof}
	Clearly, $\Lambda_a<\infty$ follows from the definition of $\mathcal{N}_a$ and Poincar\'e inequality.  Let us see $\lambda_a>0$ holds.
	
	Taking $q\geq n-1$ such that $2<pq<2n/(n-2)$ and using interpolation inequality and Sobolev inequality, we have 
	\begin{equation}
	\label{eq:inter-Sobo}
	||u||_{pq}^{2p}\leq ||u||^{2p\theta}||u||_{2n/(n-2)}^{2p(1-\theta)}\leq C_{n,p}||\nabla u||^{2p(1-\theta)}||u||^{2p\theta}
	\end{equation}
	for $u\in H_0^1(\Omega)$, where $\theta=1-\frac{n(pq-2)}{2pq}>0$.  Due to (\ref{eq:estimation-f-p}), we get $z(u)|u|^{p-1}u\in L^q(\Omega)$. If $u\in\mathcal{N}$, then by H\"older inequality we find 
	\begin{equation}
	\label{eq:Cauchy-z}
	||\nabla u||^2=\int_{\Omega}z(u)|u|^{p-2}uu\,dx\leq ||z(u)|u|^{p-1}u||_q\cdot||u||_{q'},
	\end{equation}
	where $1<q'=q/(q-1)\leq(n-1)/(n-2)\leq 2$ for $n\geq 3$. Using (\ref{eq:estimation-f-p}) and H\"older inequality again we have
	\begin{equation}
	\label{eq:Hold-z}
	||z(u)|u|^{p-1}u||_q\cdot||u||_{q'}\leq C(\Omega,n,p)||u||_{pq}^{2p-1}||u||_{pq}= C(\Omega,n,p)||u||_{pq}^{2p}.
	\end{equation}
	Thus, (\ref{eq:inter-Sobo}), (\ref{eq:Cauchy-z}) and (\ref{eq:Hold-z}) yield
	$$||\nabla u||^{2-n(pq-2)/q}\leq C||u||^{2p\theta}$$
	for all $u\in\mathcal{N}$. Noticing $d=\inf_{u\in\mathcal{N}}J(u)>0$, we get $||\nabla u||^2\geq \frac{2pd}{p-1}>0$ by (\ref{eq:J-I}). Hence, this proves $\lambda_a>0$.  
\end{proof}

\begin{lemma}[Theorem 6 in \cite{LiuB2014}]
	\label{lem:2-select}
	Let $u_0\in L^q(\Omega), n-1\leq q<\infty, q>\frac{n}{2}(p-1)(2-\frac{1}{p})$ and $u(t)$ be the classical $L^q$-solution on $[0,T_{\max})$. Then either $T_{\max}=+\infty$ or $\lim_{t\to T_{\max}}||u(t)||_q=+\infty$.
\end{lemma}

\begin{lemma}[Lemma 2.1 in \cite{LiLiu2016}]
	\label{lem:2.1}
	Let $1<p<\frac{n+2}{n-2}$. Then there is a contant $C_{n, p,\Omega}>0$ satisfies that $||\nabla u||\geq C_{n, p,\Omega}$ for all $u\in H_0^1(\Omega)\backslash\{0\}$ and $I(u)\leq 0$.
\end{lemma}

Now, we state the proof of Theorem \ref{thm:criterions}. The key idea is to verify the sign of Nehari functional for corresponding solutions of classes of the initial value ($\mathcal{G}_0$ or $\mathcal{B}$) is invariant, which just like the arguments of potential well established for low initial energy case. 
\begin{proof}[Proof of Theorem \ref{thm:criterions}]
	Put $u(t):=S(t)u_0$ for $t\in [0,T_{\max}(u_0))$. Assume first that $u_0\in\mathcal{N}_{+}$ satisfies $||u_0||\leq\lambda_{J(u_0)}$. We claim that $u(t)\in\mathcal{N}_{+}$ for all $t\in [0,T_{\max}(u_0))$ by contradiction. If there is $s$ such that $u(t)\in\mathcal{N}_{+}$ for $0\leq t<s$ and $u(s)\in\mathcal{N}$, notice that $\frac{d}{dt}||u(t)||^2=-2I(u(t))<0$ for $0\leq t<s$, then we get 
	$$||u(s)||<||u_0||\leq\lambda_{J(u_0)},\quad J(u(s))\leq J(u_0).$$
	This contradicts the definition of $\lambda_{J(u_0)}$ and proves the claim. Thus, for $t\in [0,T_{\max}(u_0))$, using (\ref{eq:J-I}) and $u(t)\in\mathcal{N}_{+}$ we have
	$$\frac{p-1}{2p}||\nabla u(t)||^2\leq J(u(t))<J(u_0),$$
	which shows that the solution $u(t)$ remains bounded in $H_0^1(\Omega)$. For those $q$ satisfies $n-1\leq q\leq \frac{2n}{n-2}$ ($n=3$ or $4$), and  $\frac{n}{2}(p-1)(2-\frac{1}{p})<q<\frac{2n}{n-2}$, we have $||u(t)||_{L^q(\Omega)}$ is bounded, by using the Sobolev inequality.  Hence applying Lemma \ref{lem:2-select}, we deduce $T_{\max}(u_0)=\infty$. Now for every $\varphi\in\omega(u_0)$, we obtain
	$$||\varphi||<||u_0||\leq\lambda_{J(u_0)}\quad \mathrm{and}\quad J(\varphi)<J(u_0).$$
	Consequently, $\varphi\not\in\mathcal{N}$ follows by definition of $\lambda_{J(u_0)}$. On the other hand, from Theorem \ref{thm:equilibria}, we know $\varphi$ is an equilibrium, which implies $I(\varphi)=0$. Hence $\varphi=0$ holds only in such situation, showing that $u_0\in\mathcal{G}_0$ as asserted.
	
	Next, assume that $u_0\in\mathcal{N}_{-}$ satisfies $||u_0||\geq\Lambda_{J(u_0)}$. By a similar argument as above, we see that $u(t)\in\mathcal{N}_{-}$ for all $t\in [0,T_{\max}(u_0))$. Now, suppose to contrary that $T_{\max}(u_0)=\infty$. Since $\frac{d}{dt}||u(t)||^2=-2I(u(t))>0$ for $t\geq 0$, for every $\varphi\in\omega(u_0)$ we infer that
	$$||\varphi||>||u_0||\geq\Lambda_{J(u_0)}\quad \mathrm{and}\quad J(\varphi)<J(u_0).$$ 
	Hence, we obtain $\varphi\not\in\mathcal{N}$ by definition of $\Lambda_{J(u_0)}$. In addition, we have $||\nabla u(t)||\geq C$ for any $t>0$ from Lemma \ref{lem:2.1}, which implies $||\nabla\varphi||\geq C$. Thus we get $0\not\in\omega(u_0)$ implying $\omega(u_0)=\varnothing$. It contradicts the assumption that $T_{\max}(u_0)=\infty$. Consequently, we conclude that $T_{\max}(u_0)<\infty$ i.e. $u_0\in\mathcal{B}$.  
\end{proof}

In order to prove Theorem \ref{thm:blow-up for M}, we give a criterion for blow-up as follows.

\begin{lemma}
	\label{lem:blowup-N-}
	Assume $2<p<p_{nl}$. If $u_0\in C(\bar{\Omega})\cap H_0^1(\Omega)\setminus\{0\}$ satisfies 
	\begin{equation}
	\label{eq:blow-up}
	||u_0||^{2p}\geq \frac{2p}{p-1}|\Omega|^{p-2}\gamma^{n-2}J(u_0),
	\end{equation}
	where $\gamma=diam (\Omega):=\sup_{x,y\in\Omega}|x-y|$. Then $u_0\in\mathcal{N}_{-}\cap\mathcal{B}$. 
\end{lemma}

\begin{proof}
	First, by using the definition of $\gamma$, we know that
	$$
	\int_{\Omega}z(u)|u|^p\,dx:=\int_{\Omega}\left(\int_{\Omega}\frac{|u(y)|^p}{|x-y|^{n-2}}\,dy\right)|u(x)|^p\,dx\geq\gamma^{2-n}||u||_p^{2p}, \forall u\in H_0^1(\Omega).
	$$
	The H\"older inequality gives 
	\begin{equation}
	\label{eq:Hold-2p}
	|\Omega|^{p-2}\gamma^{n-2}\int_{\Omega}z(u)|u|^p\,dx \geq|\Omega|^{p-2}||u||_p^{2p}>||u||^{2p}, \forall u\in H_0^1(\Omega).
	\end{equation}
	For each $u_0$ satisfies (\ref{eq:blow-up}), we have
	\begin{equation}
	\label{eq:J-u}
	||u_0||^{2p}\geq\frac{2p}{p-1}|\Omega|^{p-2}\gamma^{n-2}\left(\frac12||\nabla u_0||^2-\frac{1}{2p}\int_{\Omega}z(u_0)|u_0|^p\,dx\right)
	\end{equation}
	Combining (\ref{eq:Hold-2p}) with (\ref{eq:J-u}), we deduce that $\int_{\Omega}z(u_0)|u_0|^p\,dx>||\nabla u_0||^2$, namely, $u_0\in\mathcal{N}_{-}$. 
	
	Furthermore, by H\"older inequality, for any $u\in\mathcal{N}_{J(u_0)}$, we have
	$$|\Omega|^{2-p}\gamma^{2-n}||u||^{2p}\leq\gamma^{2-n}||u||_p^{2p}\leq\int_{\Omega}z(u)|u|^p\,dx=||\nabla u||^2\leq \frac{2p}{p-1}J(u_0).$$
	Hence, we immediately infer that 
	$$\Lambda_{J(u_0)}^{2p}\leq \frac{2p}{p-1}|\Omega|^{p-2}\gamma^{n-2}J(u_0).$$
	Therefore, if $u_0$ fulfills (\ref{eq:blow-up}), then $||u_0||\geq\Lambda_{J(u_0)}$ and Theorem \ref{thm:criterions} shows that $u_0\in\mathcal{B}$. 
\end{proof}

We now give the proof of Theorem \ref{thm:blow-up for M}.
\begin{proof}[Proof of Theorem \ref{thm:blow-up for M}]
	Let $M>0$, and let $\Omega_1,\Omega_2$ be two arbitrary disjoint open subdomains of $\Omega$. We choose an arbitrary nonzero function $v\in H_0^1(\Omega_1)\subset H_0^1(\Omega)$. Since 
	$$
	J(\alpha v)=\frac{1}{2}\alpha^2||\nabla v||^2-\frac{1}{2p}\alpha^{2p}\int_{\Omega}z(v)|v|^p\,dx, \forall \alpha \in \mathbb{R}.
	$$ 
	It follows that for sufficiently large $\alpha$
	$$||\alpha v||^{2p}\geq \frac{2p}{p-1}|\Omega|^{p-2}\gamma^{n-2}M,\quad\mathrm{and}\quad J(\alpha v)\leq 0.$$ 
	Then there exists $w\in H_0^1(\Omega_2)$ satisfying 
	$$X(w):=J(w)-\frac{1}{p}\int_{\Omega\times \Omega}\frac{|w(y)|^p|\alpha v(x)|^p}{|x-y|^{n-2}}dxdy= M-J(\alpha v).$$
	
	In fact, for $\alpha>0$, pick a function $\phi_k\in C_0^1(\Omega_2)$ such that
	\begin{equation}
	\label{eq:norm-w-k}
	||\nabla\phi_k||\geq k, \quad ||\phi_k||_\infty\leq c_0.
	\end{equation}
	Applying Hardy-Littlewood-Sobolev inequality and H\"older inequality, we find that 
	\begin{align}
	\int_{\Omega\times \Omega}\frac{|\phi_k(y)|^p|\phi_k(x)|^p}{|x-y|^{n-2}}dxdy&\leq C_{n,p}||\phi_k||_{\frac{2np}{n+2}}^{2p} \notag\\
	&\leq C_{n,p} c_0^{2p}|\Omega_2|^{\frac{n+2}{n}}.\label{eq:integration-w-k}
	\end{align}
	Similarly, we also have
	\begin{align}
	\int_{\Omega\times \Omega}\frac{|\phi_k(y)|^p|\alpha v(x)|^p}{|x-y|^{n-2}}dxdy&\leq C_{n,p,\Omega}||\phi_k||_{\frac{2np}{n+2}}^p||\alpha v||_{\frac{2np}{n+2}}^p\notag\\
	&\leq C_{n,p,\Omega}c_0^p|\Omega_2|^{\frac{n+2}{2n}}||\alpha v||_{\frac{2np}{n+2}}^p. \label{eq:int-w-v}
	\end{align}
	A direct computation shows
	\begin{eqnarray*}
		X(\phi_k)&=& J(\phi_k)-\frac{1}{p}\int_{\Omega\times \Omega}\frac{|\phi_k(y)|^p|\alpha v(x)|^p}{|x-y|^{n-2}}dxdy\\
		&=& \frac{1}{2}||\nabla \phi_k||^2-\frac{1}{2p}\int_{\Omega\times \Omega}\frac{|\phi_k(y)|^p|\phi_k(x)|^p}{|x-y|^{n-2}}dxdy-\frac{1}{p}\int_{\Omega\times \Omega}\frac{|\phi_k(y)|^p|\alpha v(x)|^p}{|x-y|^{n-2}}dxdy\\
		&\geq &  \frac{1}{2}||\nabla \phi_k||^2-C_{n,p} c_0^{2p}|\Omega_2|^{\frac{n+2}{n}}-C_{n,p,\Omega}c_0^p|\Omega_2|^{\frac{n+2}{2n}}||\alpha v||_{\frac{2np}{n+2}}^p.
	\end{eqnarray*}
	By choosing $k>0$ large enough, we obtain $w=\phi_k$.

	Now, denote $u_M:=\alpha v+w$, then 
	$$J(u_M)=J(\alpha v)+J(w)-\frac{1}{p}\int_{\Omega\times \Omega}\frac{|w(y)|^p|\alpha v(x)|^p}{|x-y|^{n-2}}dxdy= M$$ 
	and 
	$$||u_M||^{2p}\geq ||\alpha v||^{2p}\geq \frac{2p}{p-1}|\Omega|^{p-2}\gamma^{n-2}M= \frac{2p}{p-1}|\Omega|^{p-2}\gamma^{n-2}J(u_M).$$
	Thus, by Lemma \ref{lem:blowup-N-} we have $u_M\in\mathcal{N}_{-}\cap\mathcal{B}$. This completes our proof.
\end{proof}


\end{document}